\documentclass[reqno]{amsart} 
\usepackage{tikz}
\usetikzlibrary{arrows.meta,calc}
\usepackage[a4paper, total={6.2in, 9.6in}, hmarginratio=1:1]{geometry}
\usepackage[utf8]{inputenc}
\usepackage{amsmath,amsfonts,amsthm,amssymb,amsopn,mathtools}
\usepackage{aliascnt} 
\usepackage[final, pdftex, pdfpagelabels, pdfstartview = {FitH}, bookmarks, colorlinks, plainpages = false, linktoc=all, linkcolor=red, citecolor=blue, urlcolor=blue,filecolor=black]{hyperref}
\usepackage[capitalize,nameinlink]{cleveref}
\usepackage{xcolor}
\usepackage[textwidth=20mm]{todonotes}
\usepackage{pdfrender, tikz-cd, rotating}
\usepackage{fancyhdr}
\usepackage{bm}
\usepackage{enumitem}
\usepackage{etoolbox, array, setspace}
\usepackage{url}
\usepackage{comment}
\usepackage{subcaption}

\numberwithin{equation}{section}
\mathtoolsset{showonlyrefs}

\theoremstyle{plain}
\newtheorem{theorem}{Theorem}[section]
\crefname{theorem}{Theorem}{Theorems}

\newtheorem{maintheorem}{Theorem}

\newtheorem{mainproposition}{Proposition}

\newaliascnt{lemma}{theorem}
\newtheorem{lemma}[lemma]{Lemma}
\aliascntresetthe{lemma}
\crefname{lemma}{Lemma}{Lemmas}

\newaliascnt{proposition}{theorem}
\newtheorem{proposition}[proposition]{Proposition}
\aliascntresetthe{proposition}
\crefname{proposition}{Proposition}{Propositions}

\newaliascnt{corollary}{theorem}
\newtheorem{corollary}[corollary]{Corollary}
\aliascntresetthe{corollary}
\crefname{corollary}{Corollary}{Corollaries}

\newaliascnt{conjecture}{theorem}

\aliascntresetthe{conjecture}
\crefname{conjecture}{Conjecture}{Conjectures}

\newaliascnt{claim}{theorem}

\aliascntresetthe{claim}
\crefname{claim}{Claim}{Claims}

\theoremstyle{definition}

\newaliascnt{definition}{theorem}
\newtheorem{definition}[definition]{Definition}
\aliascntresetthe{definition}
\crefname{definition}{Definition}{Definitions}

\newaliascnt{example}{theorem}
\newtheorem{example}[example]{Example}
\aliascntresetthe{example}
\crefname{example}{Example}{Examples}

\newaliascnt{remark}{theorem}
\newtheorem{remark}[remark]{Remark}
\aliascntresetthe{remark}
\crefname{remark}{Remark}{Remarks}

\newcommand{\pol}{\mathbb{K}[\mathbf{x}]}
\DeclareMathOperator{\Vol}{Vol}
\DeclareMathOperator{\conv}{conv}
\newcommand{\rec}[1]{\Delta (#1)}
\newcommand{\recD}[1]{\mathcal{D} (#1)}

\newcommand{\myauthor}[3]{%
  \author{#1}
  \address{#2}
  \email{#3}
}

\myauthor
  {Tristram Bogart}
  {Departamento de Matem\'aticas, Universidad de Los Andes, Colombia}
  {tc.bogart22@uniandes.edu.co}

\myauthor
  {Federico Castillo}
  {Facultad de Matem\'aticas, Pontificia Universidad Cat\'olica de Chile, Chile}
  {federico.castillo@uc.cl}

  \myauthor
  {Damian de la Fuente}
  {Instituto de Matem\'aticas, Universidad de Talca, Talca, Chile\\   LAMFA, Universit\'e de Picardie Jules Verne, Amiens, France}
  {damiandlfa@gmail.com}

\myauthor
  {David Plaza}
  {Instituto de Matem\'aticas, Universidad de Talca, Talca, Chile}
  {dplaza@utalca.cl}

\title{A differential characterization of volume polynomials of permutohedra}

\date{\today}



\begin{document}

\begin{abstract}
We study the space $\rec{M}$ of polynomials $p(\mathbf{x}) \in \mathbb{K}[x_1, \dots, x_n]$ associated to a square $n\times n$ matrix $M$.
It consists of those $p$ for which, for each $i = 1, \dots, n$, the finite difference $p(\mathbf{x} - M_i) - p(\mathbf{x})$ is independent of the coordinate $x_i$, where $M_i$ denotes the $i$-th row of $M$.
We show that $\rec{M}$ coincides with the space defined using directional derivatives along the rows $M_i$, instead of the finite differences. 
We prove that $\rec{M}$ is graded and, moreover, is finite-dimensional if and only if every principal minor of $M$ is nonzero.
In the finite-dimensional case, $\rec{M}$ has dimension $\binom{n}{d}$ in degree $d$ for $0 \le d \le n$ and zero in higher degrees, giving total dimension $2^n$.
In the special case where $M$ is the Cartan matrix of an irreducible root system $\Phi$, we construct a basis of $\rec{M}$ consisting of volume polynomials of faces of $W(\Phi)$-permutohedra. 
This yields an elementary criterion for testing whether a polynomial admits an expression as a linear combination of these volume polynomials; we call this property geometricity.
\end{abstract}

\maketitle

\section{Introduction}

A classical result in convex geometry is that the volume of a Minkowski sum is a polynomial in the scaling parameters \cite[Chapter IV]{ewald1996combinatorial}. 
More precisely, let $K_1,\dots,K_n$ be convex bodies in $\mathbb{R}^n$, and let $\lambda_1,\dots,\lambda_n$ be nonnegative scalars. 
Then the function
\[
\Vol(\lambda_1 K_1 + \cdots + \lambda_n K_n)
\]
agrees with a homogeneous polynomial of degree $n$. 
We refer to such polynomials as \emph{volume polynomials}.
Volume polynomials are widely studied; for instance, they are important examples of Lorentzian polynomials \cite{branden2020lorentzian}.

Our guiding example is the classical permutohedron $\Pi(z)$, defined as the convex hull of all permutations of the entries of a vector $z \in \mathbb{R}^{n+1}$. 
The volume of $\Pi(z)$ is a polynomial function in the coordinates of $z$. 
One way to see this is to define $\lambda_i = z_{i} - z_{i+1}$ and observe that the translation of $\Pi(z)$ by $-z_{n+1}(1,1,\ldots,1)$ decomposes as a Minkowski sum of the form
\[
\Pi(z) -z_{n+1}\mathbf{1}_{n+1} = \lambda_1 \Pi(\mathbf{1}_1) + \cdots + \lambda_n \Pi(\mathbf{1}_n),
\]
where $\mathbf{1}_i$ denotes the vector whose first $i$ entries equal $1$ and the remaining entries equal $0$ \cite[Section 16]{postnikov2009permutohedra}.
Note that the $\lambda_i$'s are the lengths of the edges of the permutohedron.
Let $\mathbf{x}$ denote the variables $(x_1,\ldots,x_n)$.
To distinguish the function from the polynomial, we denote by $V(\mathbf{x}) \in \mathbb{R}[\mathbf{x}]$ the polynomial such that $V(\mathbf{x})\mid_{x_i=\lambda_i}=\Vol(\Pi(z))$.

The heart of the present paper is an elementary characterization of this volume polynomial. 
\begin{mainproposition}
\label{prop:heart}
The polynomial $V(\mathbf{x})\in\mathbb{R}[\mathbf{x}]$ is the \textbf{unique} (up to a scalar) homogeneous polynomial of degree $n$ such that, for each $i$, the difference $p(\mathbf{x}-\alpha_i) - p(\mathbf{x})$ is independent of the variable $x_i$. 
Here $\alpha_i$ denotes the $i$-th row of the Cartan matrix of type $A_n$, and for a vector $v$ we write $\mathbf{x} - v$ to mean $(x_1 - v_1, \dots, x_n - v_n)$. 
\end{mainproposition}

For example, in \cref{fig:difference} the difference between two permutohedra (hexagons in this case) is highlighted. 
In \Cref{Fig1} we start with a dominant vector $\lambda$ and consider the difference between the permutahedra $\Pi(\lambda) $ and $\Pi(\lambda - \alpha_1)$, while in \Cref{Fig2} we consider the difference between the permutahedra $\Pi(\lambda +h \varpi_1) $ and $\Pi(\lambda + h\varpi_1 - \alpha_1)$. 
We notice that the volume difference in both cases is the same. 
In other words, this difference does not depends on the first coordinate of $\lambda$.

\begin{figure}[ht]
    \centering
    \begin{subfigure}{0.45\textwidth}
        \centering
        \resizebox{\linewidth}{!}{\begin{tikzpicture}[scale=1, line cap=round, line join=round]

\coordinate (O) at (0,0);

\coordinate (w1) at (-0.45,0.78);
\coordinate (w2) at ( 0.45,0.78);

\coordinate (a1) at (-1.35,0.78);
\coordinate (a2) at ( 1.35,0.78);

\pgfmathsetmacro{\ax}{3}
\pgfmathsetmacro{\bx}{1.55}

\pgfmathsetmacro{\woneX}{-0.45}
\pgfmathsetmacro{\woneY}{0.78}
\pgfmathsetmacro{\wtwoX}{ 0.45}
\pgfmathsetmacro{\wtwoY}{0.78}

\pgfmathsetmacro{\lamX}{\ax*\woneX + \bx*\wtwoX}
\pgfmathsetmacro{\lamY}{\ax*\woneY + \bx*\wtwoY}


\coordinate (P0) at ({\ax*\woneX + \bx*\wtwoX},               {\ax*\woneY + \bx*\wtwoY});
\coordinate (P1) at ({(-\ax)*\woneX + (\ax+\bx)*\wtwoX},     {(-\ax)*\woneY + (\ax+\bx)*\wtwoY});
\coordinate (P2) at ({(-\ax-\bx)*\woneX + \ax*\wtwoX},       {(-\ax-\bx)*\woneY + \ax*\wtwoY});
\coordinate (P3) at ({(-\bx)*\woneX + (-\ax)*\wtwoX},        {(-\bx)*\woneY + (-\ax)*\wtwoY});
\coordinate (P4) at ({\bx*\woneX + (-\ax-\bx)*\wtwoX},       {\bx*\woneY + (-\ax-\bx)*\wtwoY});
\coordinate (P5) at ({(\ax+\bx)*\woneX + (-\bx)*\wtwoX},     {(\ax+\bx)*\woneY + (-\bx)*\wtwoY});

\pgfmathsetmacro{\h}{1}
\pgfmathsetmacro{\ap}{\ax - 2*\h}
\pgfmathsetmacro{\bp}{\bx + \h}

\coordinate (Q0) at ({\ap*\woneX + \bp*\wtwoX},               {\ap*\woneY + \bp*\wtwoY});
\coordinate (Q1) at ({(-\ap)*\woneX + (\ap+\bp)*\wtwoX},     {(-\ap)*\woneY + (\ap+\bp)*\wtwoY});
\coordinate (Q2) at ({(-\ap-\bp)*\woneX + \ap*\wtwoX},       {(-\ap-\bp)*\woneY + \ap*\wtwoY});
\coordinate (Q3) at ({(-\bp)*\woneX + (-\ap)*\wtwoX},        {(-\bp)*\woneY + (-\ap)*\wtwoY});
\coordinate (Q4) at ({\bp*\woneX + (-\ap-\bp)*\wtwoX},       {\bp*\woneY + (-\ap-\bp)*\wtwoY});
\coordinate (Q5) at ({(\ap+\bp)*\woneX + (-\bp)*\wtwoX},     {(\ap+\bp)*\woneY + (-\bp)*\wtwoY});

\fill[gray!18] (P0)--(P1)--(P2)--(P3)--(P4)--(P5)--cycle;
\fill[white]   (Q0)--(Q1)--(Q2)--(Q3)--(Q4)--(Q5)--cycle;

\fill[gray!30] (P0)--(P1)--(Q1)--(Q0)--cycle;

\draw[thick] (P0)--(P1)--(P2)--(P3)--(P4)--(P5)--cycle;
\draw[thick,dashed] (Q0)--(Q1)--(Q2)--(Q3)--(Q4)--(Q5)--cycle;

\foreach \P in {P0,P1,P2,P3,P4,P5}
  \filldraw[white,draw=black,line width=0.8pt] (\P) circle (0.05);

\foreach \Q in {Q0,Q1,Q2,Q3,Q4,Q5}
  \fill (\Q) circle (0.05);

\node[above=2pt] at (P0) {$\lambda$};
\node[above right=1pt] at (Q0) {$\lambda -\alpha_1$};

\def\L{4.2}
\draw[gray!60] (-\L,0) -- (\L,0);
\draw[gray!60] ({-0.5*\L},{0.8660254*\L}) -- ({0.5*\L},{-0.8660254*\L});
\draw[gray!60] ({-0.5*\L},{-0.8660254*\L}) -- ({0.5*\L},{0.8660254*\L});

\draw[blue, line width=1.2pt, -{Latex[length=3.5mm,width=2.5mm]}] (O) -- (a1);
\draw[red,  line width=1.2pt, -{Latex[length=3.5mm,width=2.5mm]}] (O) -- (a2);

\draw[black, line width=1.5pt, -{Latex[length=3.5mm,width=2.5mm]}] (O) -- (w1);
\draw[black, line width=1.5pt, -{Latex[length=3.5mm,width=2.5mm]}] (O) -- (w2);

\filldraw[white,draw=black,line width=1pt] (O)  circle (0.07);

\node[blue, above left=1pt]  at (a1) {$\alpha_1$};
\node[red,  above right=1pt] at (a2) {$\alpha_2$};

\node[above=1pt] at (w1) {$\varpi_1$};
\node[above=1pt] at (w2) {$\varpi_2$};

\end{tikzpicture}}
        \caption{$\Pi(\lambda) $ and $\Pi(\lambda - \alpha_1)$}
        \label{Fig1}
    \end{subfigure}
    \hfill
    \begin{subfigure}{0.45\textwidth}
        \centering
        \resizebox{\linewidth}{!}{\begin{tikzpicture}[scale=1, line cap=round, line join=round]

\coordinate (O) at (0,0);

\coordinate (w1) at (-0.45,0.78);
\coordinate (w2) at ( 0.45,0.78);

\coordinate (a1) at (-1.35,0.78);
\coordinate (a2) at ( 1.35,0.78);

\pgfmathsetmacro{\ax}{4}
\pgfmathsetmacro{\bx}{1.55}

\pgfmathsetmacro{\woneX}{-0.45}
\pgfmathsetmacro{\woneY}{0.78}
\pgfmathsetmacro{\wtwoX}{ 0.45}
\pgfmathsetmacro{\wtwoY}{0.78}

\pgfmathsetmacro{\lamX}{\ax*\woneX + \bx*\wtwoX}
\pgfmathsetmacro{\lamY}{\ax*\woneY + \bx*\wtwoY}


\coordinate (P0) at ({\ax*\woneX + \bx*\wtwoX},               {\ax*\woneY + \bx*\wtwoY});
\coordinate (P1) at ({(-\ax)*\woneX + (\ax+\bx)*\wtwoX},     {(-\ax)*\woneY + (\ax+\bx)*\wtwoY});
\coordinate (P2) at ({(-\ax-\bx)*\woneX + \ax*\wtwoX},       {(-\ax-\bx)*\woneY + \ax*\wtwoY});
\coordinate (P3) at ({(-\bx)*\woneX + (-\ax)*\wtwoX},        {(-\bx)*\woneY + (-\ax)*\wtwoY});
\coordinate (P4) at ({\bx*\woneX + (-\ax-\bx)*\wtwoX},       {\bx*\woneY + (-\ax-\bx)*\wtwoY});
\coordinate (P5) at ({(\ax+\bx)*\woneX + (-\bx)*\wtwoX},     {(\ax+\bx)*\woneY + (-\bx)*\wtwoY});

\pgfmathsetmacro{\h}{1}
\pgfmathsetmacro{\ap}{\ax - 2*\h}
\pgfmathsetmacro{\bp}{\bx + \h}

\coordinate (Q0) at ({\ap*\woneX + \bp*\wtwoX},               {\ap*\woneY + \bp*\wtwoY});
\coordinate (Q1) at ({(-\ap)*\woneX + (\ap+\bp)*\wtwoX},     {(-\ap)*\woneY + (\ap+\bp)*\wtwoY});
\coordinate (Q2) at ({(-\ap-\bp)*\woneX + \ap*\wtwoX},       {(-\ap-\bp)*\woneY + \ap*\wtwoY});
\coordinate (Q3) at ({(-\bp)*\woneX + (-\ap)*\wtwoX},        {(-\bp)*\woneY + (-\ap)*\wtwoY});
\coordinate (Q4) at ({\bp*\woneX + (-\ap-\bp)*\wtwoX},       {\bp*\woneY + (-\ap-\bp)*\wtwoY});
\coordinate (Q5) at ({(\ap+\bp)*\woneX + (-\bp)*\wtwoX},     {(\ap+\bp)*\woneY + (-\bp)*\wtwoY});

\fill[gray!18] (P0)--(P1)--(P2)--(P3)--(P4)--(P5)--cycle;
\fill[white]   (Q0)--(Q1)--(Q2)--(Q3)--(Q4)--(Q5)--cycle;

\fill[gray!30] (P0)--(P1)--(Q1)--(Q0)--cycle;

\draw[thick] (P0)--(P1)--(P2)--(P3)--(P4)--(P5)--cycle;
\draw[thick,dashed] (Q0)--(Q1)--(Q2)--(Q3)--(Q4)--(Q5)--cycle;

\foreach \P in {P0,P1,P2,P3,P4,P5}
  \filldraw[white,draw=black,line width=0.8pt] (\P) circle (0.05);

\foreach \Q in {Q0,Q1,Q2,Q3,Q4,Q5}
  \fill (\Q) circle (0.05);

\node[above=2pt] at (P0) {$\lambda + h\varpi_1$};
\node[above right=1pt] at (Q0) {$\lambda +h\varpi_1 -\alpha_1 $};

\def\L{4.2}
\draw[gray!60] (-\L,0) -- (\L,0);
\draw[gray!60] ({-0.5*\L},{0.8660254*\L}) -- ({0.5*\L},{-0.8660254*\L});
\draw[gray!60] ({-0.5*\L},{-0.8660254*\L}) -- ({0.5*\L},{0.8660254*\L});

\draw[blue, line width=1.2pt, -{Latex[length=3.5mm,width=2.5mm]}] (O) -- (a1);
\draw[red,  line width=1.2pt, -{Latex[length=3.5mm,width=2.5mm]}] (O) -- (a2);

\draw[black, line width=1.5pt, -{Latex[length=3.5mm,width=2.5mm]}] (O) -- (w1);
\draw[black, line width=1.5pt, -{Latex[length=3.5mm,width=2.5mm]}] (O) -- (w2);

\filldraw[white,draw=black,line width=1pt] (O)  circle (0.07);

\node[blue, above left=1pt]  at (a1) {$\alpha_1$};
\node[red,  above right=1pt] at (a2) {$\alpha_2$};

\node[above=1pt] at (w1) {$\varpi_1$};
\node[above=1pt] at (w2) {$\varpi_2$};

\end{tikzpicture}}
        \caption{ $\Pi(\lambda + h\varpi_1) $ and $\Pi(\lambda + h\varpi_1 - \alpha_1)$}
        \label{Fig2}
    \end{subfigure}
    \caption{Difference  between two permutohedra }
    \label{fig:difference}
\end{figure}
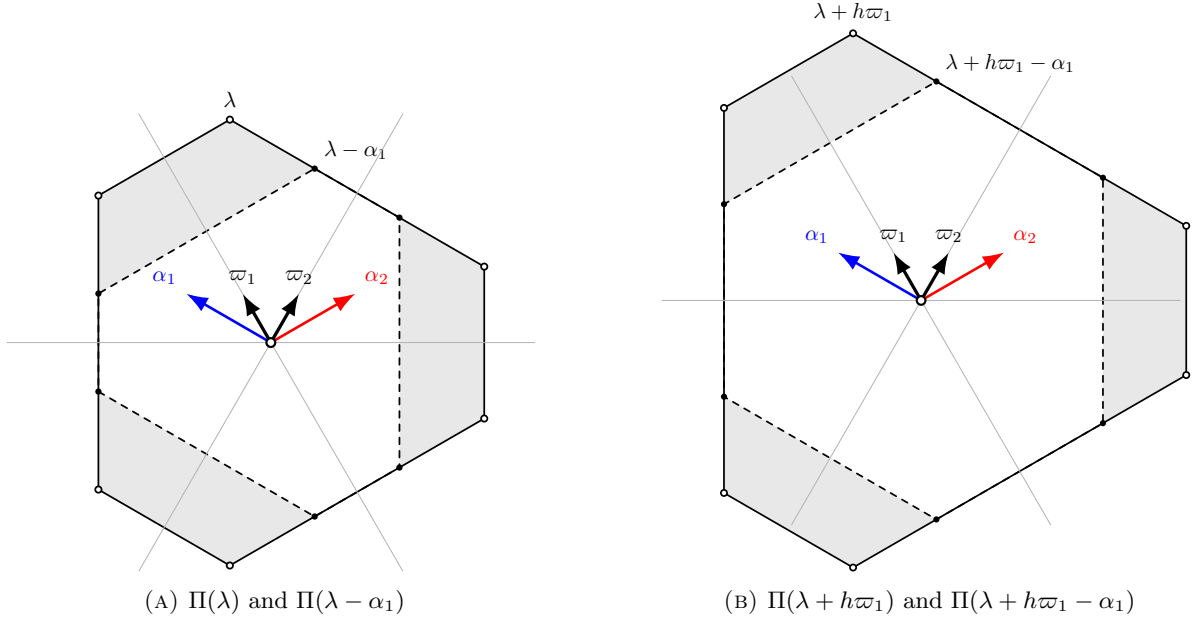

\Cref{prop:heart} is valid for any Cartan matrix of a root system $\Phi$, replacing the volume polynomial $V(\mathbf{x})$ of the permutohedron with the volume polynomial $V^\Phi(\mathbf{x})$ of the appropriate $W(\Phi)$-permutohedron, see \cref{cor: V is determined by system}.
Our main theorem characterizes the linear space spanned by the volume polynomials of all the faces of the $W(\Phi)$-permutohedron, including $V^\Phi(\mathbf{x})$.
While proving this, we observed that the key arguments do not rely on the specific structure of Cartan matrices, but only on certain properties of their principal minors.
This led us to develop a more general theory for arbitrary square matrices $M$, which we now describe.
Let $\mathbb{K}$ be a field of characteristic zero. 

\begin{definition}\label{def: rec M}
Let $M \in M_n(\mathbb{K})$ be a square matrix. 
For each $i\in[n]$, let $M_i$ denote the $i$-th row of $M$.
Let $\Delta_i:\pol \rightarrow \pol$ denote the \emph{finite difference} operator given by 
\begin{equation}
    \Delta_i (p(\mathbf{x})) = p(\mathbf{x}-M_i) - p(\mathbf{x}), 
\end{equation}
and let $\mathcal{D}_i = M_i \cdot \nabla$ denote the directional derivative in the direction $M_i$.

We define $\Delta (M) \leq \pol$ and $\mathcal{D}(M) \leq \pol$ as the subspaces
\begin{equation}
   \Delta (M) = \bigcap_{i=1}^{n}  \left\{ p \in \pol ~\middle|~ \frac{\partial}{\partial x_i}\Delta_i(p(\mathbf{x})) =0\right\}  \quad \mbox{and} \quad  \mathcal{D}(M) = \bigcap_{i=1}^{n} \left\{ p \in \pol ~\middle|~ \frac{\partial}{\partial x_i}\mathcal{D}_i(p(\mathbf{x})) =0\right\}.
\end{equation}
In other words, $\Delta(M)$ and $\mathcal{D}(M)$ are the spaces of polynomials for which the variable $x_i$ does not appear in $\Delta_i(p(\mathbf{x}))$ and $\mathcal{D}_i(p(\mathbf{x}))$, respectively, for every $i\in[n]$.
\end{definition}

Our first result shows that both spaces coincide. 

\begin{maintheorem} \label{main: equal}
    For any square matrix $M$, we have $\Delta(M)=\mathcal{D}(M)$. 
\end{maintheorem}

The main result in this paper is a characterization of this space in the case of Cartan matrices.

\begin{maintheorem}
\label{main:A}
Let $C$ be the Cartan matrix of an irreducible root system $\Phi$.
There is a basis for $\mathcal{D}(C)$ consisting of volume polynomials of faces of $W(\Phi)$-permutohedra.
\end{maintheorem}

In particular, when $\Phi$ is of type $A_n$, this gives a basis for $\mathcal{D}(C)$ consisting of volumes of faces of the classical permutohedron $\Pi(z)$.
We call a polynomial that is a linear combination of volume polynomials of faces of the permutohedron a \emph{geometric polynomial}. 
In \cite{castillo2023size}, certain polynomials that count the sizes of lower intervals in Bruhat order of affine Weyl groups are shown to be geometric. 
In \cite{castillo2025paper}, an analogous statement is conjectured for general lower intervals and is proved for lower intervals of dominant elements of type $\tilde{A}_3$.
Our main motivation for \cref{main:A} was to establish an elementary tool to prove geometricity. 

The proof of \cref{main:A} consists of two parts.
First we prove a general structure theorem for $\mathcal{D}(M)$.
We say that a subspace $L\subset\mathbb{K}[\mathbf{x}]$ is \emph{graded} if $p(\mathbf{x}) \in L$ implies $p_d(\mathbf{x}) \in L$ for all $d$, where $p_d$ is the homogeneous component of $p$ of degree $d$, and we denote by $L_d$ its graded pieces.

\begin{maintheorem}
\label{main:B}
Let $M\in M_n(\mathbb{K})$ with all principal minors nonzero. 
Then $\mathcal{D}(M)$ is a graded subspace with 
\begin{equation}
   \dim_{\mathbb{K}}  \mathcal{D}(M)_d = \begin{cases}
       \binom{n}{d}, & \mbox{if } 0\leq d \leq n;\\[5pt]
       0, & \mbox{otherwise.}
   \end{cases}
\end{equation}
In particular, $\dim_{\mathbb{K}} \mathcal{D}(M) = 2^n$.
\end{maintheorem}

The condition on the principal minors is necessary: when some principal minor vanishes, $\mathcal{D}(M)$ is infinite-dimensional (see  \Cref{prop:rec_to_quotient} and \cref{ex: greater than 2^n}).
Note, however, that $\mathcal{D}(M)$ is always graded (see \cref{prop:graded}).

To prove \cref{main:B}, we exploit a duality between $\mathcal{D}(M)$ and a quotient of the polynomial ring by the ideal generated by the corresponding second-order differential operators in $\mathbb{K}[\frac{\partial}{\partial x_1},\dots,\frac{\partial}{\partial x_n}]$.
This reduces the computation of the graded dimension of $\mathcal{D}(M)$ to the computation of the Hilbert series of this quotient ring, which we carry out using standard arguments in commutative algebra.

To complete the proof of \cref{main:A}, we show that the volume polynomials of faces of permutohedra belong to $\mathcal{D}(C)$.
The same strategy applies when $M=I_n$ is the identity matrix: a basis for $\mathcal{D}(I_n)$ is given by square-free monomials, which arise as volume polynomials of faces of a rectangular box. 
However, there exist matrices $M$ for which the polynomial in the largest degree, which is unique up to scalar, does not have all coefficients of the same sign.
Consequently, the basis for $\mathcal{D}(M)$ cannot, in general, consist of volume polynomials.

A natural further direction is to characterize volume polynomials beyond the examples considered here. 
The Lorentzian property provides a necessary algebraic condition, but it is not sufficient. 
In fact, it was recently determined for which parameters (degree and number of variables) every volume polynomial is Lorentzian \cite{menges2026comparing}. 
Characterizing the space of all volume polynomials remains a major open problem. 
However, it may be feasible to obtain such characterizations for specific families of polytopes.

\subsection{Structure of the paper}
In \Cref{sec: rec} we study the spaces $\rec{M}$ and $\recD{M}$ and prove   \Cref{main: equal} and \Cref{main:B}.
In \Cref{sec: vols}, we define the volume polynomials $V_J^\Phi(\mathbf{x})$ and establish \Cref{main:A} and \Cref{prop:heart} for general Weyl groups.

\subsection{Acknowledgments} 
We thank Nicol\'as Libedinsky for his helpful comments. 
TB was partially supported by internal grant INV-2025-213-3438 from the Faculty of Science of the University of the Andes. 
FC was partially supported by FONDECYT-ANID grant 1260970.
DP was partially supported by FONDECYT-ANID grant 1240199.
We thank the support of the MATH-AmSud program (Grant 24-MATH-01).
This project began during the visit of Damian de la Fuente to the Instituto de Matemática of the University of Talca, supported by their short-term visits program.
The authors thank the institute for its hospitality.

\section{The spaces $\rec{M}$ and $\recD{M}$} \label{sec: rec}

Throughout this section, we fix a  positive integer $n$, a field $\mathbb{K}$ of characteristic zero, and a square matrix $ M =(M_{ij})\in M_n(\mathbb{K})$.

\subsection{Equality and graded structure of $\rec{M}$ and $\recD{M}$}

In this section, we show that the spaces $\rec{M}$ and $\recD{M}$ introduced in \Cref{def: rec M} coincide. 
Furthermore, we show that they are graded spaces. 
We begin by making  the definition of the differential operator $\mathcal{D}_i$ explicit. 
\begin{definition}
 For $i\in [n]$ we  define the differential operator
\begin{equation}\label{eq:differential_forms}
	\mathcal{D}_{i} \coloneqq \sum_{j=1}^{n} M_{ij} \frac{\partial}{\partial x_{j}}.
\end{equation}
\end{definition}

The following lemma  relates the finite difference operator $\Delta_i$ and the differential operator $\mathcal{D}_i$. 

\begin{lemma} \label{lem: Delta in terms of D}
For all $p(\mathbf{x})\in \mathbb{K}[\mathbf{x}]$ and $i\in [n]$ we have
	\begin{equation}\label{eq:finite_to_differentials}
		\Delta_i(p(\mathbf{x})) = \sum_{s=1}^{d}\frac{(-1)^{s}}{s!} \mathcal{D}_{i}^s p(\mathbf{x}),
	\end{equation}
	where $d$ is the degree of the polynomial $p(\mathbf{x})$. 
\end{lemma}

\begin{proof}
    We recall that the operator $\Delta_i$ is defined as $\Delta_i(p(\mathbf{x}))=p(\mathbf{x}-M_{i}) - p(\mathbf{x})$, where $M_i $ is the $i$-th row of the matrix $M$.  
    Since both $\Delta_i$ and $\mathcal{D}_i$ are linear, without loss of generality, we can assume that the polynomial $p(\mathbf{x})$ is a monomial; say $p(\mathbf{x})=\mathbf{x}^{\beta}$ for some $\beta =(\beta_1, \ldots , \beta_n) \in \mathbb{N}^n$. 

Using the Binomial Theorem and reordering the terms we get

\begin{equation} \label{eq: Taylor A}
p(\mathbf{x} - M_i)
=
\sum_{s=0}^{|\beta|}
\ \sum_{\substack{ k\in \mathbb{N}^n \\  |k| = s \\ k\preceq \beta}}
\left(
\prod_{j=1}^n \binom{\beta_j}{k_j} (-M_{ij})^{k_j}x_j^{\beta_j-k_j}
\right)
\end{equation}
where $\preceq$ denotes the component-wise partial order on $\mathbb{N}^n$. 

On the other hand, we have 
\begin{equation}\label{eq: Taylor B}
 \mathcal{D}_i^s(p(\mathbf{x}))
=
\sum_{\substack{ k\in \mathbb{N}^n \\  |k| = s \\ k\preceq \beta}}
s!
\prod_{j=1}^n
M_{ij}^{k_j}\binom{\beta_j}{k_j}
\, x_j^{\beta_j - k_j},   
\end{equation}
which can be easily proved by induction on $s$.

Therefore, multiplying \eqref{eq: Taylor B} by $(-1)^s/s!$ and summing over all $0\leq s \leq |\beta|$ we have
\begin{equation}
 p(\mathbf{x}-M_i) = \sum_{s=0}^{|\beta|}\frac{(-1)^{s}}{s!} \mathcal{D}_{i}^s p(\mathbf{x}).
\end{equation}
Finally, by isolating the term corresponding to $s=0$ in the above sum we obtain the identity \eqref{eq:finite_to_differentials}.
\end{proof}

\begin{lemma}\label{lem:commute}
Let $p(\mathbf{x}) \in \mathbb{K}[\mathbf{x}]$, $i,j\in [n]$ and $s\geq 0$. 
	If $ \displaystyle\frac{\partial}{\partial x_{j}} p(\mathbf{x}) = 0 $  then $ \displaystyle\frac{\partial}{\partial x_{j}} \mathcal{D}^s_i p(\mathbf{x}) = 0 $.
\end{lemma}

\begin{proof}
	This follows from the fact that operators $ \mathcal{D}_i $ and the partial derivatives $\displaystyle \frac{\partial}{\partial x_{j}} $ commute. 
\end{proof}

We say that $ p $ is \textbf{free} of $ x_{j} $ when  $ \displaystyle \frac{\partial}{\partial x_{j}} p = 0 $. 

\begin{definition}
    Given a polynomial $p\in \pol$ and $m\in \mathbb{N}$ we denote by $p_m$ its $m$-th homogeneous component of degree $m$. 
    We say a subspace $U$ of $\pol$ is graded if $p\in U$ implies $p_m\in U$ for all $m\geq 0$. 
\end{definition}

\begin{proposition}\label{prop:graded}
The subspace $\rec{M}$ is graded.
\end{proposition}

\begin{proof}
Suppose that $p \in \rec{M}$ and let $d=\operatorname{deg}p$.
Fix $1\leq i\leq n$.
The degree $j$ component of \eqref{eq:finite_to_differentials} is
\begin{equation}\label{eq:degreej}
\left(\Delta_{i} p\right)_j \;=\; \sum_{s=1}^{d-j} \frac{(-1)^{s}}{s!}\, \mathcal{D}_i^{s}\, p_{j+s}.
\end{equation}
Since $\Delta_{i} p$ is free of $ x_{i} $, the same is true for each of its components $(\Delta_{i} p)_j$.
We show by descending induction on~$m$ that $\mathcal{D}_i\, p_m$ is free of $ x_{i} $ for all $0\leq m  \leq d$.
The base case is $m=d$.
From \eqref{eq:degreej} with $j = d-1$, we see that $\mathcal{D}_i\, (p_d)=-(\Delta_{i} p)_{d-1}$ is free of $ x_{i} $.

Now we proceed with the inductive step.
Let $m<d$ and suppose that $ \mathcal{D}_{i}(p_{m '}) $ is free of $ x_{i} $ for all $ m < m' \leq d   $.
By \cref{lem:commute}, we have that $ \mathcal{D}_{i}^s(p_{m '}) $ is also free of $ x_{i} $, for $ s \geq 1 $ and $ m ' > m  $.
Letting $ j = m-1 $ in \eqref{eq:degreej}, we get
\begin{equation}
\mathcal{D}_{i}(p_{m}) \;=\; -(\Delta_{i} p)_{m-1} +\sum_{s=2}^{d-m+1} \frac{(-1)^{s}}{s!}\, \mathcal{D}_i^{s}\, (p_{m-1+s} ).
\end{equation}
The sum in the right hand side is free of $ x_{i} $ by induction hypothesis.
Since $(\Delta_{i} p)_{m-1}$ is also free of $ x_{i} $, we conclude that so is $ \mathcal{D}_{i}(p_{m}) $, completing the inductive step.

Therefore, $ \mathcal{D}_{i}(p_m) $ is free of $ x_{i} $ for all $ 0\leq m \leq d$.

By \cref{lem:commute}, $ \mathcal{D}_{i}^s(p_m) $ is free of $ x_{i} $ for any $ s\geq 1 $.
Thus \eqref{eq:finite_to_differentials} applied to $ p_m $ instead of $ p $ shows that $ \Delta_i(p_{m}) $ is free of $ x_{i} $.
Since this is true for every $ i $ we have that $ p_{m} \in \rec{M}$ as we wanted to show.
\end{proof}

\begin{corollary}\label{cor:one_step}
	We have 
	\begin{equation}\label{eq:one_step_diff}
		\rec{M}=\recD{M}.
	\end{equation}
\end{corollary}

\begin{proof}
Let $p\in \recD{M}$. 
Then  $\mathcal{D}_{i} (p )$ is free of $x_i$ for all $i$.
Then, by \Cref{lem: Delta in terms of D} and   \Cref{lem:commute}, it follows that $ p \in \rec{M} $.
Conversely, suppose that $ p\in \rec{M} $.
If $p$ is homogeneous of degree $d$, note that $$\mathcal{D}_i\, (p)=-(\Delta_{i} p)_{d-1}$$ is free of $x_i$, for all $i\in[n]$.
That is, $p\in \recD{M}$.
If $p$ is not homogeneous, then its homogeneous components $p_m$ are also in $\recD{M}$ by \Cref{prop:graded} and the argument above.
Hence $p\in \recD{M}$, as desired.
\end{proof}

\subsection{Duality}

We consider the  $ \mathbb{K}$-vector spaces $ S = \mathbb{K}[\mathbf{y}] $ and $ T = \mathbb{K}[\mathbf{x}] $, with $\mathbf{y}=(y_1,\ldots,y_n)$ and $\mathbf{x}=(x_1,\ldots,x_n)$.
These spaces are naturally graded.
We let $S$ act on $T$ by 
\begin{equation}
 q\cdot p \coloneq   q \left( \frac{\partial}{\partial x_1}, \cdots, \frac{\partial}{\partial x_n} \right) p, 
\end{equation}
for $q\in S$ and $p\in T$. 
We then define the pairing
\begin{equation}
    \begin{array}{cccl}
       \langle - , - \rangle : &  S\times T & \longrightarrow  & \mathbb{K}\\
     & (q,p)   &   \longrightarrow & \langle q,p \rangle \coloneqq (q\cdot p)(\mathbf{0}). 
    \end{array}
\end{equation}

This pairing respects the grading in the sense that if $S_d$ and $T_e$ denote the homogeneous component of degree $d$ and $e$ of $S$ and $T$, respectively, then we have $ \langle  S_d ,T_e  \rangle =0 $ whenever $d\neq e$.
Furthermore, on each degree, the induced pairing $S_d\times T_d  \longrightarrow  \mathbb{K}$ is perfect, as can be checked by noting that the set of monomials give dual bases (up to scalar).

Let $I$ be a subspace of $S$.
We define
\begin{equation}
     I^{\perp}  \coloneqq \{  p\in T \mid \langle y,p \rangle =0, \, \forall y\in I   \} \quad \mbox{and } \quad
 \tilde{I}  \coloneq   \{ p\in T  \mid y\cdot p =0, \, \forall y \in I\} .  
\end{equation}

\begin{lemma}  \label{lem: graded linear}
    Let $I$ be a graded subspace of $S$. Then:
    \begin{enumerate}
        \item $I^{\perp}$ is a graded subspace of $T$.
        \item $I^{\perp}\simeq S/I$ as graded vector spaces.
        \item Moreover, if $I$ is an ideal of $S$ then $I^\perp = \tilde{I}$.
    \end{enumerate}
\end{lemma}

\begin{proof}
The first claim follows immediately from the fact that the pairing respects the grading.

We now prove the second claim. The pairing induces a bilinear map
\begin{equation}
\overline{\langle - , - \rangle} : S/I  \times  I^{\perp} \longrightarrow \mathbb{K},
\end{equation}
given by $ \overline{\langle q+I  ,  p  \rangle} \coloneqq \langle q,p \rangle$, which is well defined by the definition of $I^{\perp}$.
 Since $I$ is graded, we have $ (I^{\perp})^{\perp} = I$.
This follows from the fact that the pairing respects the grading and that each homogeneous component is finite-dimensional. 
Consequently, the induced pairing
\begin{equation}
(S/I)_d  \times (I^{\perp})_d   \longrightarrow \mathbb{K}.
\end{equation}
is perfect for every $d \geq 0$, and it also respects the grading.
It follows that $ ( I^{\perp})_d \simeq ((S/I)_d)^\ast$.
Since each homogeneous component is finite-dimensional, we also have $ ((S/I)_d)^\ast \simeq (S/I)_d$.
Therefore, $I^{\perp} \simeq S/I$ as graded vector spaces.

We now prove the third claim. 
The inclusion $\tilde{I} \subset I^{\perp}$ is trivial. 
For the other inclusion we take $p\in I^{\perp}$ and $y\in I$. 
We must show that $y\cdot p =0$. 
Let $y'\in S$. 
We have
\begin{equation}
    \langle   y' , y\cdot p \rangle= (y'\cdot (y\cdot p))(\mathbf{0}) = ((y'y)\cdot p) (\mathbf{0}) =\langle   y'y,p  \rangle  =0, 
\end{equation}
where the last equality follows from $y'y\in I$.
Since the pairing is non-degenerate, we get $y\cdot p=0$ and thus $p\in \tilde{I}$.
Therefore, $I^{\perp} \subset \tilde{I}$ and the third claim follows. 
\end{proof}

We now consider the ideal $ I_M \coloneqq ( y_i L_i, i\in [n]) $ of $S$, where $ L_i = \sum_{j=1}^{n} M_{ij}y_j \in S $ are the linear forms obtained from the rows of our matrix $M$.
We remark that $I_M$ is an homogeneous ideal of $ S $, and also a graded $ \mathbb{K} $-linear subspace.

\begin{proposition} \label{prop: iso graded vector spaces}
The spaces $\rec{M}$ and $S/I_M$ are isomorphic as graded $\mathbb{K}$-vector spaces. 
\end{proposition}

\begin{proof}

By \cref{cor:one_step}  we have $\rec{M} = \recD{M}$. 
On the other hand, by definition, $\recD{M}$ consists of the polynomials $p \in T$ such that $y_iL_i\cdot p=0$ for all $i$.
Since these elements generate the ideal $I_M$, it follows that $y \cdot p = 0$ for all $y \in I_M$. 
In other words, $\recD{M} = \widetilde{I_M}$.
Therefore, we get $\rec{M}= \widetilde{I_M}$. 
With this equality at hand, the result follows by a direct application of \Cref{lem: graded linear}. 
\end{proof}

\subsection{Commutative algebra}
We now turn to some standard tools from commutative algebra, and we present the arguments in an elementary and self-contained manner.

Thanks to \Cref{prop: iso graded vector spaces} we can understand the graded vector space structure of $\Delta(M)$ through the ideal $I_M$.
We will do so by studying its zero locus in the affine space $\mathbb{A}^n_{\mathbb{K}}$.

\begin{lemma} \label{lem: minors to zero}
The variety $V(I_M)$ is equal to $\{ \mathbf{0}\}$ if and only if all principal minors of $M$ are nonzero. 
\end{lemma}

\begin{proof}
For a subset $J \subseteq [n]$, we denote by $M_J$ the principal submatrix of $M$ indexed by $J$, that is, the submatrix obtained by restricting $M$ to the rows and columns indexed by $J$.

Assume that some principal minor of $M$ is zero. Then there exists a nonempty subset $J \subseteq [n]$ and a nonzero vector
$a=(a_j)_{j\in J}\in \mathbb{K}^J$
such that
$M_J\cdot a=\mathbf{0}$.
We define the vector $b=(b_1,\dots,b_n)\in \mathbb{K}^n$ by
\begin{equation}
b_j=
\begin{cases}
a_j,& \text{if } j\in J,\\
0,& \text{if } j\notin J.
\end{cases}
\end{equation}
We stress that $b\neq \mathbf{0}$ since $a\neq \mathbf{0}$ and $J\neq \emptyset$.
Since $M_J\cdot a=\mathbf{0}$ it is easy to check that $ b_i\sum_{j=1}^n M_{ij}b_j=0,$
for all $i\in [n]$. 
Hence, $\mathbf{0} \neq b\in V(I_M)$.

Conversely, we assume that all principal minors of $M$ are nonzero. Let $a=(a_1,\dots,a_n)\in V(I_M)$ and $J=\{i\in [n]\mid a_i\neq 0\}$.
Since $I_M=(y_iL_i,\; i\in [n])$, we have
\begin{equation}
\sum_{j\in J} M_{ij}a_j=0, \quad \forall i\in J.
\end{equation}
Thus, we have a homogeneous system with matrix of coefficients $M_J$.
By assumption, $M_J$ is invertible. So, the unique solution of the system is the zero vector.
We conclude that $J=\emptyset$, or equivalently, $a=(0,\dots,0)$.
Hence, $V(I_M)=\{\mathbf{0}\}$.
\end{proof}

\begin{proposition}\label{prop:rec_to_quotient}
  The quotient ring $S/I_M$ is finite-dimensional over $\mathbb{K}$ if and only if all principal minors of $M$ are nonzero.

  Equivalently, $\rec{M} $ is finite-dimensional over $\mathbb{K}$ if and only if all principal minors of $M$ are nonzero.   
\end{proposition}

\begin{proof}
Suppose that \(M\) has some principal minor equal to zero. If $n=1$ then $M$ is the zero ideal and \(S/I_M \simeq S\) is infinite-dimensional, so assume $n > 1$. \Cref{lem: minors to zero} implies that there exists some \(\mathbf{0} \neq a \in V(I_M)\). 
Let \(\ell = \mathbb{K}a\) be the line passing through \(a\) and the origin. 
As \(I_M\) is homogeneous, we have \(\ell \subset V(I_M)\).
Since $n > 1$, we can choose a linear form \(L \in S\) such that the hyperplane \(V(L)\) does not contain the line \(\ell\).

We claim that the set \(\{L^k + I_M\}_{k \in \mathbb{N}}\) is linearly independent in \(S/I_M\).
Suppose that there exist some \(m \in \mathbb{N}\) and scalars \(c_k \in \mathbb{K}\) such that
\begin{equation}
    \sum_{k=1}^m c_k(L^k+I_M) = I_M.
\end{equation}
Let \(f \coloneqq \sum_{k=1}^m c_k L^k \in I_M\). 
Hence, \(f\) vanishes on \(V(I_M)\), and in particular on \(\ell\). 
Since \(L\) is linear, for all \(t \in \mathbb{K}\), we have
\begin{equation}
    0=f(ta)= \sum_{k=1}^m (c_kL(a)^k)t^k.
\end{equation}
Since \(\operatorname{char}(\mathbb{K})=0\) and thus \(|\mathbb{K}|=\infty\), we conclude that \(c_k L(a)^k = 0\) for all \(1 \leq k \leq m\). 
But \(L(a) \neq 0\). 
Therefore, \(c_k = 0\) for all \(1 \leq k \leq m\).
Consequently, \(\{L^k + I_M\}_{k \in \mathbb{N}}\) is linearly independent in \(S/I_M\), and \(\dim_{\mathbb{K}} S/I_M = \infty\). 

Conversely,  suppose that all the principal minors of $M$ are nonzero. 
Let $\overline{\mathbb{K}}$ be an algebraic closure of $\mathbb{K}$ and set $I_{M,\overline{\mathbb{K}}}:= I_M \otimes_\mathbb{K} \overline{\mathbb{K}}[\mathbf{y}]$.
Applying \Cref{lem: minors to zero} over $\overline{\mathbb{K}}$, we get $ V_{\overline{\mathbb{K}}}(I_{M,\overline{\mathbb{K}}})=\{\mathbf{0}\}$.

By Hilbert's Nullstellensatz, $ V_{\overline{\mathbb{K}}}(I_{M,\overline{\mathbb{K}}})=\{\mathbf{0}\}$ implies $ \sqrt{\smash[b]{I_{M,\overline{\mathbb{K}}}}}=(y_1,\dots,y_n)$.
Hence, for each $i\in [n]$ there exists $r_i\geq 1$ such that $ y_i^{r_i}\in I_{M,\overline{\mathbb{K}}}$.
It follows that the quotient $\overline{\mathbb{K}}[\mathbf{y}]/I_{M,\overline{\mathbb{K}}}$ is finite-dimensional over $\overline{\mathbb{K}}$.
On the other hand, there is an isomorphism of $\overline{\mathbb{K}}$-vector spaces
\begin{equation}
\overline{\mathbb{K}}[\mathbf{y}]/I_{M,\overline{\mathbb{K}}}
\cong
\bigl(S/I_M\bigr)\otimes_{\mathbb{K}} 
\overline{\mathbb{K}}.
\end{equation}
Therefore, $S/I_M$ is finite-dimensional over $\mathbb{K}$ as we wanted to show.
The second statement is now a direct consequence of \Cref{prop: iso graded vector spaces}. 
\end{proof}

\begin{example} \label{ex: greater than 2^n}
    Let \begin{equation}
    M=    \left(\begin{array}{rrr}
0 & 1 & 0 \\
0 & 0 & 1 \\
1 & 0 & 0
\end{array}\right).
\end{equation}
The only non-zero principal minor comes from $M$ itself.
The zero locus of the ideal $(y_1y_2, y_2y_3, y_1y_3)$ is the union of the three coordinate axes in $\mathbb{A}_{\mathbb{K}}^3$, that is,
\[
V(y_1y_2, y_2y_3, y_1y_3)
= V(y_1,y_2)\ \cup\ V(y_2,y_3)\ \cup\ V(y_1,y_3).
\]
The ring $\mathbb{K}[y_1,y_2, y_3]/(y_1y_2,y_2y_3,y_1y_3)$ has infinite dimension over $\mathbb{K}$ since the infinite set of nonzero monomials $y_i^n$, for $i=1,2,3$ and $n\geq 1$, is linearly independent.
It follows that $\rec{M}$ is infinite-dimensional. 
\end{example}

\subsection{Dimension count}
We now compute the dimension of each homogeneous component of $\rec{M}$ in the case where $M$ has all its principal minors nonzero. 
By \cref{prop: iso graded vector spaces}, this reduces to computing the dimensions of the graded components of $S/I_M$, that is, its Hilbert series
\begin{equation}
\text{Hilb}(S/I_M ; t) = \sum_{n=0}^\infty \dim_{\mathbb{K}} \left( S/I_M\right)_n t^n \in \mathbb{Q}[[t]].
\end{equation}

When the quotient $S/I_M$ is finite-dimensional, as is the case when $M$ has all its principal minors nonzero by \Cref{prop:rec_to_quotient}, this series is in fact a polynomial.

In this case, to compute the Hilbert series, we observe that $I_M$ is a complete intersection.
In particular, its generators $y_1L_1, \dots, y_nL_n$ form a regular sequence. 
We may therefore use the Koszul complex to obtain a graded free resolution of $S/I_M$, from which the Hilbert series can be read off (see \cite[Chapter I.16]{peeva2010graded}). 
While our approach follows this general strategy, we aim to keep the arguments as elementary as possible.

\begin{remark}
Alternatively, one can compute $\operatorname{Hilb}(S/I_M ; t)$ directly by exhibiting a basis.
It is straightforward to show that the squarefree monomials span $S/I_M$ as a $\mathbb{K}$-vector space.
The key step for linear independence is proving that $y_1 \cdots y_n \notin I_M$, which can be done by a rather technical induction; for this reason, we prefer the approach via regular sequences.
One could also attempt a Gröbner basis computation, but the initial ideal (under standard orders such as grevlex or lex) depends on the matrix $M$, even when all principal minors are nonzero, so this route does not yield a uniform argument.
\end{remark}

We compute the Hilbert series of $S/I_M$ by modding out by one generator at a time, for which we recall the following standard definitions.
A \emph{homogeneous system of parameters} is a sequence of homogeneous elements $f_1,\dots, f_n \in S$ such that $S/(f_1,\dots, f_n)$ is a finite-dimensional vector space.
A \emph{graded regular sequence} is a homogeneous system of parameters with the property that $f_1$ is  not a zero divisor of $S$ and $f_i$ is not a zero divisor in $S/(f_1,\dots,f_{i-1})$ for $i>1$.

\begin{lemma}\label{lem:regular}
   If all the principal minors of $M$ are nonzero then $y_1L_1, \ldots , y_nL_n$ is a graded regular sequence. 
\end{lemma}

\begin{proof}
    By \cref{prop:rec_to_quotient}, we have that $y_1L_1, \ldots , y_nL_n$ is an homogeneous system of parameters in $S$.
    Clearly $y_1,\dots y_n$ is a graded regular sequence in $S$.
    If a graded ring has a graded regular sequence, then any homogeneous system of parameters is regular \cite[Theorem 5.9]{stanley1996combinatorics} and \cite[Theorem 2.3.1]{sturmfels1993invariant}, so we can conclude that $y_1L_1, \ldots , y_nL_n$ is regular.
\end{proof}

The following proposition is well-known \cite[Exercise 16.4]{peeva2010graded}.
We include the proof for the reader's convenience.
\begin{proposition}\label{prop:exercise}
    Let $f_1, \dots, f_k$ be a graded regular sequence in $S$ with $\deg f_i = d_i$.
    Then
\begin{equation}
 \operatorname{Hilb}(S/(f_1, \dots, f_k) ; t) = \frac{\prod_{i=1}^k(1-t^{d_i})}{(1-t)^n}.
\end{equation}
\end{proposition}
\begin{proof}
    We argue by induction on the length of the regular sequence.
    For $k=1$, multiplication by $f_1$ gives an exact sequence
    \[
0 \to S \xrightarrow{\cdot f_1} S \to S/(f_1) \to 0.
    \]
    The first map sends degree $m$ to degree $m + d_1$, so in each degree $m$ we get
    \[
    \dim (S/(f_1))_m = \dim S_m - \dim S_{m-d_1},
    \]
    and therefore $\operatorname{Hilb}(S/(f_1);t) = (1-t^{d_1})/(1-t)^n$.
    For the inductive step, since $f_i$ is a nonzero divisor in $S/(f_1,\dots,f_{i-1})$, we have the exact sequence
    \[
0 \to S/(f_1,\dots,f_{i-1}) \xrightarrow{\cdot f_i} S/(f_1,\dots,f_{i-1}) \to S/(f_1,\dots, f_i) \to 0.
    \]
    The same degree-counting argument gives
    \[
    \operatorname{Hilb}(S/(f_1,\dots, f_i);t) = (1-t^{d_i})\operatorname{Hilb}(S/(f_1,\dots, f_{i-1});t),
    \]
    and the result follows by the inductive hypothesis.
\end{proof}

We can now prove one of our main results.

\begin{theorem} \label{thm: main dimension}
If all principal minors of $M$ is nonzero, then $\rec{M}$ is a graded subspace with
\begin{equation}
   \dim_{\mathbb{K}}  \rec{M}_d = \begin{cases}
       \binom{n}{d}, & \mbox{if } 0\leq d \leq n;\\[5pt]
       0, & \mbox{otherwise.}
   \end{cases}
\end{equation}
In particular, $\dim_{\mathbb{K}} \rec{M} = 2^n$.
\end{theorem}
\begin{proof}
    By \cref{lem:regular}, the elements $y_1L_1, \dots, y_nL_n$ form a graded regular sequence.
    Since each generator has degree two, \cref{prop:exercise} gives
\[
\operatorname{Hilb}(S/I_M ; t) = \frac{(1-t^2)^n}{(1-t)^n} = (1+t)^n = \sum_{d=0}^n \binom{n}{d} t^d.
\]
The result now follows from \cref{prop:rec_to_quotient} and \Cref{prop: iso graded vector spaces}.
\end{proof}

\section{Volumes and the Permutohedron}\label{sec: vols}

In this section, we apply \Cref{thm: main dimension} to the case where $M$ 
is the Cartan matrix associated to an irreducible root system, and show that the polynomials given by the volumes of the faces of the permutahedron associated to $M$ form a basis of $\rec{M}$. 

We follow the notation and conventions of \cite{humphreys1992reflection}.
Let $\Phi$ be an irreducible (reduced, crystallographic) root system of rank $n$ and let $C$ be its corresponding Cartan matrix. 
Let $E$ be the ambient (real) Euclidean space spanned by $\Phi$, with inner product ${\langle-,-\rangle: E \times E \to \mathbb{R} }$.

We fix a set ${\Delta=\{\alpha_i\mid i\in [n]\}}$ of simple roots and let $\Phi^+$ denote the corresponding set of positive roots.
Let ${\alpha^\vee=2\alpha/\langle\alpha,\alpha\rangle}$ be the coroot corresponding to ${\alpha\in\Phi}$. 
The fundamental weights ${\varpi_i}$ are defined by the equations ${\langle\varpi_i,\alpha_j^\vee\rangle=\delta_{ij}}$. 
They form a basis of $E$. 

Let $H_{\alpha}$ be the hyperplane of $E$ orthogonal to a root $\alpha$, and let $s_\alpha$ denote the reflection through $H_{\alpha}$. 
For $\alpha_i\in \Delta$ we write $s_i=s_{\alpha_i}$. 
The group $W$ of orthogonal transformations of $E$ generated by $S=\{s_i\mid i\in [n]\}$ is the finite Weyl group of $\Phi$.

\begin{definition}
For any $\lambda \in E$, the \emph{$W$-permutohedron} $\Pi_W(\lambda)$ is the convex hull of the $W$-orbit of $\lambda$:
\begin{equation}
    \Pi_W(\lambda) \coloneqq \conv \{ w \cdot \lambda \mid w \in W \}.
\end{equation}
It is full-dimensional unless $\lambda=\mathbf{0}$, in which case $\Pi_W(\mathbf{0})=\{\mathbf{0}\}$.
\end{definition}
For general information about $W$-permutohedra see \cite{hohlweg2012permutahedra}.
By \cite[Chapter 16]{postnikov2009permutohedra}, the $W$-permutohedron admits a decomposition as a Minkowski sum of fundamental polytopes:
\begin{equation}\label{eq:minkowski_decomposition}
\Pi_W(x_1\varpi_1+\cdots +x_n\varpi_n ) = x_1 \Pi_W(\varpi_1) + x_2 \Pi_W(\varpi_2) + \dots + x_n \Pi_W(\varpi_n),
\end{equation}
for all $x_i\geq0$.

\begin{remark}
    It is a classic result that there is a unique (up to scalar) translation invariant Lebesgue measure in $\mathbb{R}^n$, which we normally call \emph{volume}.
    In \cref{rem:normalized_volume} we point out how to choose the scalar multiples of the volumes.
    For now, the specific choice is not crucial.
\end{remark}

Recall from the introduction that there exists a homogeneous polynomial $V^\Phi(\mathbf{x}) \in \mathbb{R}[\mathbf{x}]$ of degree $n$, called the \emph{volume polynomial}, with the property that
\begin{equation}\label{eq: vol pol W}
V^\Phi(\lambda_1,\ldots,\lambda_n) \coloneqq \Vol \Pi_W\left(\sum_{i=1}^n \lambda_i\varpi_i\right), \quad \forall \lambda_i\geq0.
\end{equation}

For any subset $J \subseteq [n]$, let $W_J \subseteq W$ be the parabolic subgroup generated by $\{s_j \mid j \in J\}$. 
These subgroups are associated with root subsystems $\Phi_J =\operatorname{Span}_\mathbb{Z}\{\Delta_J\}\cap\Phi$, which may be reducible.
If the Dynkin diagram of $\Phi_J$ has connected components $J_1, \dots, J_k$, then the parabolic subgroup decomposes as a direct product $W_J \cong W_1 \times \dots \times W_k$, where $W_i=W_{J_i}$ is the Weyl group of the irreducible root system $\Phi_i=\Phi_{J_i}$.
Geometrically, this implies that the polytope $\Pi_{W_J}(\lambda) \coloneqq\conv \{ w \cdot \lambda \mid w \in W_J \}$ is a Cartesian product of smaller $W$-permutohedra:
\begin{equation}\label{eq:cartesian_product}
\Pi_{W_J}(\lambda) = \Pi_{W_1}(\lambda) \times \dots \times \Pi_{W_k}(\lambda).
\end{equation}
We then define
\begin{equation}\label{eq:volume_polynomials}
V_J^\Phi(\mathbf{x})\coloneqq V^{\Phi_1}(\mathbf{x})\cdots V^{\Phi_k}(\mathbf{x}),
\end{equation}
where each polynomial on the right hand side is a volume polynomial \eqref{eq: vol pol W}.
Here, we are using the canonical inclusion $\mathbb{R}[x_j\mid j\in J_i]\hookrightarrow \mathbb{R}[\mathbf{x}]$, to see each $ V^{\Phi_i}$ as a polynomial in $\mathbb{R}[\mathbf{x}]$.

Because the volume of a Cartesian product is the product of the volumes of its factors, it is clear that $V_J^\Phi(\mathbf{x})$ coincides with the $|J|$-dimensional volume of $ \Pi_{W_J}\left(\sum_{i=1}^n x_i \varpi_i\right)$, when all  $x_i\geq0$.

We collect these facts in the following definition.

\begin{definition}\label{def: vols}
Given a subset $J\subset [n]$, we define the homogeneous polynomial $V_J^\Phi(\mathbf{x})\in\mathbb{R}[\mathbf{x}]$ of degree $|J|$, as in \eqref{eq:volume_polynomials}.
They are polynomials in the variables $\{x_j \mid j\in J\}$.
We call these polynomials \emph{volume polynomials}.
Notice that $V^\Phi(\mathbf{x}) = V_{[n]}^\Phi(\mathbf{x})$.
We define
\begin{equation}
\mathcal{V}^\Phi\coloneqq \{ V_J^\Phi(\mathbf{x}) \mid J\subset [n]\}.
\end{equation}
\end{definition}

\begin{remark}\label{rem:normalized_volume}
    Consider the weight lattice $\Lambda \coloneqq \mathbb{Z}\varpi_1 \oplus \cdots \oplus \mathbb{Z}\varpi_n \subset E$.
    A linear space $L \subset E$ is said to be $\Lambda$-rational if $L$ can be generated by elements of $\Lambda$.
    An affine linear space is $\Lambda$-rational if it is the translation of a $\Lambda$-rational linear subspace.
    When $L$ is $\Lambda$-rational, the lattice $\Lambda \cap L$ is of full rank within $L$.
    We consider the volume on $L$ to be normalized so that the covolume\footnote{the covolume of a lattice is the volume of any fundamental parallelepiped.} of the lattice $\Lambda \cap L$ is one.
    In the present case, the polytopes $\Pi_{W_J}(\lambda)$, with $\lambda \in E$ are not full dimensional in $V$, but they are full dimensional in the $\Lambda$-rational affine subspace
    \[
    L = \operatorname{Span}_\mathbb{R}\{\alpha_j \mid j\in J\} + \lambda.
    \]
    We use the lattice $\Lambda$ to normalize the volumes of polytopes that are not full dimensional in \eqref{eq:volume_polynomials}.
\end{remark}

The following lemma shows that the volume polynomials are linearly independent (cf. \cite[Lemma 4.7, Corollary 4.8]{castillo2023size}).
We write $[p]_\beta$ for the coefficient of the monomial $\mathbf{x}^\beta$ in $p \in \mathbb{R}[\mathbf{x}]$.
Given $\beta\in\mathbb{N}^n$, we let $\mathsf{supp}(\beta)\coloneqq\{i\in[n]\mid \beta_i\neq0\}$.

\begin{lemma}\label{lem: vols are li}
Let $J\subset [n]$, and let $\beta\in\{0,1\}^n$.
Then
\begin{equation}
\left[ V_J^\Phi(\mathbf{x}) \right]_\beta \mbox{ is }
\begin{cases}
>0, \mbox{ if } \mathsf{supp}(\beta) =J\\
0, \mbox{ otherwise}.
\end{cases}
\end{equation}
In particular, the set $\mathcal{V}^\Phi$ is linearly independent. 
\end{lemma}

We now aim to show that $\mathcal{V}^\Phi$ is a basis of $\rec{C}$. 
Notice that by \Cref{def: vols} and \Cref{lem: vols are li}, we have that $\mathcal{V}^\Phi\subset\mathbb{R}[\mathbf{x}]$ is homogeneous, linearly independent, and its component
\begin{equation*}
(\mathcal{V}^\Phi)_d = \{ V_J^\Phi(\mathbf{x}) \in \mathcal{V}^\Phi \mid |J|=d\}
\end{equation*}
consists of $\binom{n}{d}$ distinct polynomials.
Therefore, by \Cref{thm: main dimension},  in order to prove that $\mathcal{V}^\Phi$ is a basis for $\rec{C}$ it suffices to show the following:
\begin{enumerate}
    \item $\mathcal{V}^\Phi\subset \Delta(C) = \mathcal{D}(C)$, and
    \item All the principal minors of $C$ are nonzero. 
\end{enumerate}

Let us focus in the first condition. 
We first address the top-degree volume polynomial $V^\Phi(\mathbf{x})=V_{[n]}^\Phi(\mathbf{x})$.

\begin{proposition}
\label{prop: GeomSatisfy}
Let $\Phi$ be an irreducible root system of rank $n$ and let $C=(c_{ij})$ be its corresponding Cartan matrix. 
Then, $V^{\Phi}(\mathbf{x}) \in \rec{C}$. 
\end{proposition}

\begin{proof}
Let $i\in [n]$.
By \Cref{cor:one_step} it is enough to show that $V^{\Phi}(\mathbf{x}) \in \recD{C}$.
Equivalently, we must show that the directional derivative of $V^{\Phi}(\mathbf{x})$ with respect to $\alpha_i$ is free of the variable $x_i$.
For $j \in [n]$, we define $W_j \coloneqq W_{[n] \setminus \{j \}}$.
By \cite[Proposition 18.6]{postnikov2009permutohedra},
for each $k$ one has
\begin{equation}
\frac{\partial}{\partial x_k} V^\Phi(x_1,\dots,x_n)
=
\sum_{j=1}^n
[W:W_j]\,
\frac{\langle\varpi_k,\varpi_j \rangle}{\langle\alpha_j,\varpi_j\rangle}\,
V_{[n]\setminus \{j\}}^\Phi(\mathbf{x}).  
\end{equation}

Since $\alpha_i=c_{1 i}\,\varpi_1+\cdots+c_{ni}\varpi_n$, we obtain 

\begin{align}
\frac{\partial}{\partial \alpha_{i}} V^{\Phi}(\mathbf{x})  &=  \sum_{k=1}^n  c_{ki}\frac{\partial}{\partial x_k} V^\Phi(\mathbf{x}) \\[5pt]
&= \sum_{k=1}^n c_{ki} \sum_{j=1}^n
[W:W_j]\,
\frac{\langle\varpi_k,\varpi_j \rangle}{\langle\alpha_j,\varpi_j\rangle}\,
V_{[n]\setminus \{j\}}^\Phi(\mathbf{x}) \\[5pt]
&=    \sum_{j=1}^n [W:W_j]\, \sum_{k=1}^n  c_{ki} 
\frac{\langle\varpi_k,\varpi_j \rangle}{\langle\alpha_j,\varpi_j\rangle}\,
V_{[n]\setminus \{j\}}^{\Phi}(\mathbf{x}) \\[5pt]
&=   \sum_{j=1}^n [W:W_j]\,  
\frac{\langle\alpha_i,\varpi_j \rangle}{\langle\alpha_j,\varpi_j\rangle}\,
V_{[n]\setminus \{j\}}^{\Phi}(\mathbf{x}) \\[5pt]
&=     [W:W_i]\,  
V_{[n]\setminus \{i\}}^{\Phi}(\mathbf{x}) \\[5pt]
\end{align}

In particular, we see that $ \frac{\partial}{\partial \alpha_{i}} V^{\Phi}(\mathbf{x}) $ does not depend on the variable $x_i$, as we wanted to show.  
Given that $i$ was arbitrary, we conclude that $V^{\Phi}(\mathbf{x}) \in \recD{C} =\rec{C}$.
\end{proof}

\begin{corollary} \label{cor: GeomSatisfy} 
We have $\mathcal{V}^\Phi\subset\rec{C}$.
\end{corollary}

\begin{proof}
Let $J\subset[n]$.
We must show that $V_J^\Phi\in\rec{C}$.
Let $J_1,\ldots, J_k$ be the connected components of the Dynkin diagram of $\Phi_J$, so that
\begin{equation}
V_J^\Phi(\mathbf{x})= V^{\Phi_1}(\mathbf{x})\cdots V^{\Phi_k}(\mathbf{x}),
\end{equation}
where $\Phi_i=\Phi_{J_i}$, as in \eqref{eq:volume_polynomials}.

For $i\in[k]$, let $C_i\coloneqq C_{J_i}$ be the principal submatrix of $C$ obtained by deleting rows and columns not in $J_i$.
Recall that $V^{\Phi_i}(\mathbf{x})$ is a polynomial in the variables $x_j$ with $j\in J_i$.
Therefore,
\begin{equation*}
V_J^{\Phi}(\mathbf{x})\in \rec{C} \iff V^{\Phi_i}(\mathbf{x})\in \rec{C_i}, \ \forall i\in[k].
\end{equation*}

Let $i\in[k]$.
As $\Phi_i$ is an irreducible root system with Cartan matrix $C_i$, we get that $V^{\Phi_i}(\mathbf{x})\in \rec{C_i}$ by \Cref{prop: GeomSatisfy}.
Therefore $V_J^{\Phi}(\mathbf{x})\in \rec{C}$, concluding the proof.
\end{proof}

\begin{remark}
We sketch an alternative proof to \Cref{prop: GeomSatisfy}, showing that the volume polynomial is in $\Delta(C)$ directly.
Let $C^+\coloneqq\operatorname{Span}_{\mathbb{R}_{\geq0}}\{\varpi_1,\ldots,\varpi_n\}$ denote the dominant cone.
Let $i\in[n]$, and let $\lambda$ such that $\lambda-\alpha_i\in C^+$.
The difference set between $\Pi_W (\lambda)$ and $\Pi_W (\lambda-\alpha_i)$ is a disjoint union of $[W:W_i]$ convex polytopes, each with the same volume.
One of these polytopes $\Pi$ has vertices
\begin{equation}
\label{eq:vertices}
\left\{ w\cdot \lambda ~\mid~ w\in W_i\right\} \bigsqcup \left\{ w\cdot (\lambda-\alpha_i) ~\mid~ w \in W_i \right\}.
\end{equation}
Consider the polytope
\begin{equation}\label{eq:missing_piece}
\Gamma \coloneqq \conv \{\mathbf{0} \cup \{ -w\cdot \alpha_i ~\mid~ w \in W_i\} \}.
\end{equation}
Then we have the following Minkowski decomposition
\begin{equation}
\label{eq:killer_decomposition}
\Pi = \Pi_{W_i}(\lambda) + \Gamma,
\end{equation}
which is easily verified using the inequality description of $\Pi_{W_i}(\lambda)$.
Therefore, since the volume of $\Pi_{W_i}(\lambda_1\varpi_1+\cdots+\lambda_n\varpi_n)$ is a polynomial in the coordinates $\lambda_j$ with $j\neq i$, the same is true for $\Pi$ by \eqref{eq:killer_decomposition}.
\end{remark}

The last missing piece to establish \Cref{main:A} is the following.

\begin{lemma}\label{lem: Cartan minors}
Let $C$ be  the Cartan matrix associated to an irreducible root system of rank $n$.
Then every principal minor of $C$ is positive.
\end{lemma}

\begin{proof}
Following Bourbaki's conventions \cite{Bour46}, let $S\coloneqq CD$, where $D$ is the diagonal matrix given by $D_{ii}=\dfrac{\langle\alpha_i,\alpha_i\rangle}{2}$.
Then $S_{ij}=\langle\alpha_i,\alpha_j\rangle$, so that $S$ is a symmetric and positive-definite matrix.
Therefore, each of its principal minors is positive.
Let $I\subset [n]$ and consider the principal submatrix $S_I=D_IC_I$.
Since $\operatorname{det}(S_I)$ and $\operatorname{det}(D_I)$ are both positive, we conclude that $\operatorname{det}(C_I)$ must be positive as well.
\end{proof}

We can now finally conclude the proof of our main theorem, which we restate for the reader's convenience.

\begin{theorem}\label{thm: main}
Let $\Phi$ be an irreducible root system and let $C$ be its Cartan matrix. 
Then $\mathcal{V}^\Phi$ is a basis of $\rec{C}$, and its dimension as a real vector space is $2^n$.
In particular, $\rec{C}$ is graded, its graded dimension is $(1+z)^n$ and $(\mathcal{V}^\Phi)_d$ is a basis of $\rec{C}_d$.
\end{theorem}

\begin{proof}
All the claims in the theorem follow by combining \Cref{thm: main dimension}, \Cref{cor: GeomSatisfy} and \Cref{lem: Cartan minors}. 
\end{proof}

\begin{corollary}\label{cor: V is determined by system}
Up to scalar, the polynomial $V^\Phi(\mathbf{x})$ is the unique homogeneous polynomial $p\in\mathbb{R}[\mathbf{x}]$ of degree $n$ with the property that
\begin{equation}\label{eq: vol rec}
p(\mathbf{x}) - p(\mathbf{x}-\alpha_i) \mbox{ does not depend on } x_i \ \forall i\in[n], 
\end{equation}
where the vector coordinates of $\alpha_i$ are taken in the fundamental weight basis.
\end{corollary}

\begin{proof}
By \Cref{thm: main}, we know that $V^\Phi(\mathbf{x})\in \Delta(C)$, so it satisfies \eqref{eq: vol rec}.
On the other hand, suppose such a $p$ exists.
Then $p\in \Delta(C)_n$.
By \Cref{thm: main}, we have that $(\mathcal{V}^\Phi)_n = \{V^\Phi(\mathbf{x})\}$ is a basis of $\Delta(C)_n$.
\end{proof}

We observe that finding the coordinates in the volume basis is straightforward.
Recall that $\beta\mapsto \mathsf{supp}(\beta)$ gives a bijection between elements of $\{0,1\}^n$ and subsets of $[n]$.

\begin{proposition}
Let $p \in \Delta(C)$, with $C$ an irreducible Cartan matrix of finite type of rank $n$.
Then
\begin{equation}
p(\mathbf{x}) = \sum_{\beta \in\{0,1\}^n} \dfrac{[p]_\beta}{[V^\Phi_{\mathsf{supp}(\beta)}(\mathbf{x})]_\beta} V^\Phi_{\mathsf{supp}(\beta)}(\mathbf{x}) .
\end{equation}
\end{proposition}

\begin{proof}
By \Cref{thm: main}, we have that there exist some $\mu_\beta\in\mathbb{R}$ such that
\begin{equation*}
p(\mathbf{x})= \sum_{\beta \in\{0,1\}^n} \mu_\beta V^\Phi_{\mathsf{supp}(\beta)}(\mathbf{x}).
\end{equation*}
Let $\beta \in\{0,1\}^n$.
Then \Cref{lem: vols are li} implies that $[p]_\beta=\mu_\beta \,[V^\Phi_{\mathsf{supp}(\beta)}]_\beta$ with $[V^\Phi_{\mathsf{supp}(\beta)}]_\beta>0$.
\end{proof}

We close by remarking that the squarefree coefficient of the volume polynomial of the whole permutohedron in type $A_n$ is known to be $n!$ \cite[Theorem 16.3 (7)]{postnikov2009permutohedra}.

\bibliographystyle{amsalpha} 
\bibliography{ref}

\providecommand{\bysame}{\leavevmode\hbox to3em{\hrulefill}\thinspace}
\providecommand{\MR}{\relax\ifhmode\unskip\space\fi MR }
\providecommand{\MRhref}[2]{%
  \href{http://www.ams.org/mathscinet-getitem?mr=#1}{#2}
}
\providecommand{\href}[2]{#2}
\begin{thebibliography}{CdlFLP25}

\bibitem[BH20]{branden2020lorentzian}
Petter Br{\"a}nd{\'e}n and June Huh, \emph{Lorentzian polynomials}, Annals of Mathematics \textbf{192} (2020), no.~3, 821--891.

\bibitem[Bou02]{Bour46}
Nicolas Bourbaki, \emph{Lie groups and lie algebras}, Springer Berlin Heidelberg, 2002.

\bibitem[CdlFLP23]{castillo2023size}
Federico Castillo, Damian de~la Fuente, Nicolas Libedinsky, and David Plaza, \emph{On the size of bruhat intervals}, arXiv preprint arXiv:2309.08539 (2023).

\bibitem[CdlFLP25]{castillo2025paper}
\bysame, \emph{Paper boat}, arXiv preprint arXiv:2504.04489 (2025).

\bibitem[Ewa96]{ewald1996combinatorial}
G{\"u}nter Ewald, \emph{Combinatorial convexity and algebraic geometry}, vol. 168, Springer Science \& Business Media, 1996.

\bibitem[Hoh12]{hohlweg2012permutahedra}
Christophe Hohlweg, \emph{Permutahedra and associahedra: generalized associahedra from the geometry of finite reflection groups}, Associahedra, {T}amari lattices and related structures, Progr. Math., vol. 299, Birkh\"auser/Springer, Basel, 2012, pp.~129--159.

\bibitem[Hum92]{humphreys1992reflection}
James~E Humphreys, \emph{Reflection groups and coxeter groups}, no.~29, Cambridge university press, 1992.

\bibitem[Men26]{menges2026comparing}
Amelie Menges, \emph{Comparing the sets of volume polynomials and lorentzian polynomials}, Combinatorics, Probability and Computing \textbf{35} (2026), no.~1, 26--39.

\bibitem[Pee10]{peeva2010graded}
Irena Peeva, \emph{Graded syzygies}, vol.~14, Springer Science \& Business Media, 2010.

\bibitem[Pos09]{postnikov2009permutohedra}
Alexander Postnikov, \emph{Permutohedra, associahedra, and beyond}, International Mathematics Research Notices \textbf{2009} (2009), no.~6, 1026--1106.

\bibitem[Sta96]{stanley1996combinatorics}
Richard~P Stanley, \emph{Combinatorics and commutative algebra}, Springer, 1996.

\bibitem[Stu93]{sturmfels1993invariant}
Bernd Sturmfels, \emph{Invariant theory of finite groups}, Algorithms in Invariant Theory, Springer, 1993, pp.~25--75.

\end{thebibliography}

\end{document}